\documentclass[a4paper,10pt,twoside]{article}
\usepackage{amsmath,amssymb,amsthm,dsfont}
\usepackage[english]{babel}
\usepackage[bookmarks]{hyperref} 
\usepackage{mathrsfs} 
\usepackage{color,soul} 

\usepackage{geometry}
\newtheorem{theorem}{Theorem}[section]

\newtheorem{lem}[theorem]{Lemma} 
\newtheorem{prop}[theorem]{Proposition}
\newtheorem{coro}[theorem]{Corollary} 
\theoremstyle{definition}
\newtheorem{rem}[theorem]{Remark}

\newcommand{\Z}{\mathbb Z}

\newcommand{\R}{\mathbb R}

\newcommand{\ep}{\epsilon}

\newcommand{\Ntext}{ \text{NLS}_k}
\newcommand{\Nemph}{ \emph{NLS}_k}
\newcommand{\re}[1]{\mbox{Re} \ #1} 
\newcommand{\scal}[1]{\left\langle #1 \right\rangle} 

\newcommand{\defendproof}{\hfill $\Box$} 

\usepackage{fancyhdr} 
\begin{document}
\title{\sc Global well-posedness for a $L^2$-critical nonlinear higher-order Schr\"odinger equation}
\author{\sc{Van Duong Dinh}} 
\date{ }
\maketitle

\begin{abstract}
We prove the global well-posedness for a $L^2$-critical defocusing cubic higher-order Schr\"odinger equation, namely
\[
i\partial_t u + \Lambda^k u = -|u|^2 u,
\]
where $\Lambda=\sqrt{-\Delta}$ and $k\geq 3, k \in \Z$ in $\R^k$ with initial data $u_0 \in H^\gamma, \gamma>\gamma(k):=\frac{k(4k-1)}{14k-3}$. 
\end{abstract}


\section{Introduction and main results}
\setcounter{equation}{0}
Let $k \geq 2, k \in \Z$. We consider the Cauchy problem for the defocusing cubic nonlinear higher-order Schr\"odinger equation posed on $\R^k$, namely
\[
\left\{
\begin{array}{rcl}
i\partial_t u(t,x) + \Lambda^k u(t,x)&=& -|u(t,x)|^2 u(t,x), \quad t\geq 0, x \in \R^k, \\
u(0,x) &=& u_0(x) \in H^\gamma(\R^k),
\end{array}
\right.
\tag{$\Ntext$}
\]
where $\Lambda=\sqrt{-\Delta}$ is the Fourier multiplier by $|\xi|$. When $k=2$, ($\Ntext$) corresponds to the well-known Schr\"odinger equation (see e.g. \cite{Bourgain-98}, \cite{CazenaveWeissler}, \cite{Cazenave}, \cite{Kato95}, \cite{Tao}, \cite{CoKeStaTaTao-almost-conservation}, \cite{CoKeStaTaTao-resonant}, \cite{CollianderGrillakisTzirakis}, \cite{CollianderRoy}, \cite{Dodson}, \cite{Dodson-scattering} and references therein). When $k=4$, it is the fourth-order Schr\"odinger equation take into consideration the role of small fourth-order dispersion in the propagation of intense laser beams in a bulk medium with Kerr nonlinearity (see e.g. \cite{HaoHsiaoWang06}, \cite{HaoHsiaoWang07}, \cite{Pausader}, \cite{Pausadercubic}).\newline
\indent It is worth noticing that the ($\Ntext$) is $L^2$-critical in the sense that if $u$ is a solution to ($\Ntext$) on $(-T, T)$ with initial data $u_0$, then 
\begin{align}
u_\lambda(t,x)= \lambda^{-k/2} u( \lambda^{-k} t, \lambda^{-1} x), \label{scaling}
\end{align}
is also a solution of ($\Ntext$) on $(-\lambda^k T, \lambda^k T)$ with initial data $u_\lambda(0)$ and 
\[
\|u_\lambda(0)\|_{L^2(\R^k)}=\|u_0\|_{L^2(\R^k)}.
\] 
It is known (see e.g. \cite{Cazenave}, \cite{Dinh}, \cite{Dinh-fourth-order}) that ($\Ntext$) is locally well-posed in $H^\gamma(\R^k)$ when $\gamma>0$. Moreover, these local solutions enjoy mass conservation,i.e.
\begin{align}
\|u(t)\|_{L^2(\R^k)} = \|u_0\|_{L^2(\R^k)}, \label{mass conservation}
\end{align}
and $H^{k/2}$ solutions have the conserved energy,i.e. 
\begin{align}
E(u(t)):= \int_{\R^k}  \frac{1}{2} |\Lambda^{k/2} u(t,x)|^2 +\frac{1}{4} |u(t,x)|^4 dx = E(u_0). \label{energy conservation}
\end{align}
The conservations of mass and energy combine with the persistence of regularity (see e.g. \cite{Dinh-fourth-order}) immediately yield the global well-posedness for ($\Ntext$) with initial data in $H^\gamma(\R^k)$ when $\gamma \geq k/2$. Note also (see \cite{Dinh}) that one has the local well-posedness for ($\Ntext$) when initial data $ u_0 \in L^2(\R^k)$ but the time of existence depends not only on the size but also on the profile of the initial data. In addition, if $\|u_0\|_{L^2(\R^k)}$ is small enough, then ($\Ntext$) is global well-posed and scattering in $L^2(\R^k)$. It is conjectured that ($\Ntext$) is in fact globally well-posed for initial data in $H^\gamma(\R^k)$ with $\gamma\geq 0$. This paper concerns with the global well-posedness of ($\Ntext$) in $H^\gamma(\R^k)$ when $0<\gamma <k/2$. Let us recall known results for the defocusing cubic Schr\"odinger equation in $\R^2$,i.e. ($\text{NLS}_{2}$). The first attempt to this problem due to  Bourgain in \cite{Bourgain-98} where he used a ``Fourier truncation'' approach to prove the global existence for $\gamma>3/5$. It was then improved for $\gamma>4/7$ by I-team in \cite{CoKeStaTaTao-almost-conservation}. The proof is based on the almost conservation of a modified energy functional. The idea is to replace the conserved energy $E(u)$, which is not available when $\gamma<1$, by an ``almost conserved'' quantity $E(I_Nu)$ with $N\gg 1$ where $I_N$ is a smoothing operator which behaves like the identity for low frequencies $|\xi|\leq N$ and like a fractional integral operator of order $1-\gamma$ for high frequencies $|\xi|\geq 2N$. Since $I_Nu$ is not a solution to ($\text{NLS}_{2}$), we may expect an energy increment. The key idea is to show that on the time interval of local existence, the increment of the modified energy $E(I_Nu)$ decays with respect to a large parameter $N$. This allows to control $E(I_Nu)$ on time interval where the local solution exists, and we can iterate this estimate to obtain a global in time control of the solution by means of the bootstrap argument. Fang-Grillakis then upgraded this result to $\gamma\geq 1/2$ in \cite{FangGrillakis}. Later, Colliander-Grillakis-Tzirakis improved for $\gamma>2/5$ in \cite{CollianderGrillakisTzirakis} using an almost interaction Morawetz inequality. Subsequent paper \cite{CollianderRoy} has decreased the necessary regularity to $\gamma>1/3$. Afterwards, Dodson established in \cite{Dodson} the global existence for ($\text{NLS}_2$) when $\gamma>1/4$. The proof combines the almost conservation law and an improved interaction Morawetz estimate. Recently, Dodson in \cite{Dodson-scattering} proved the global well-posedness and scattering for ($\text{NLS}_2$) for initial data $u_0 \in L^2(\R^2)$ using the bilinear estimate and a frequency localized interaction Morawetz estimate. We next recall some known results about the global well-posedness below energy space for the fourth-order Schr\"odinger equation. In \cite{Guo}, the author considered the more general fourth-order Schr\"odinger equation, namely
\[
i\partial_t u + \lambda \Delta u + \mu \Delta^2 u + \nu |u|^{2m}u =0,
\]
and established the global well-posedness in $H^\gamma(\R^n)$ for $\gamma> 1+\frac{mn-9+\sqrt{(4m-mn+7)^2+16}}{4m}$ under the assumption $4<mn<4m+2$ and of course some conditions on $\lambda, \mu$ and $\nu$. For the mass-critical fourth-order Schr\"odinger equation in high dimensions $n\geq 5$, Pausader-Shao proved in \cite{PausaderShao} that the $L^2$-solution is global and scattering under some conditions. Recently, Miao-Wu-Zhang in \cite{MiaoWuZhang} showed the global existence and scattering below energy space for the defocusing cubic fourth-order Schr\"odinger equation in $\R^n$ with $n=5,6,7$. To our knowledge, there is no result concerning the global existence (possibly scattering) for ($\text{NLS}_4$). \newline
\indent The purpose of this paper is to prove the global existence of ($\Ntext$) with $k\geq 3, k \in \Z$ below the energy space $H^{k/2}(\R^k)$. 
\begin{theorem}\label{theorem global existence}
Let $k \geq 3, k \in \Z$. The initial value problem $(\Nemph)$ is globally well-posed in $H^\gamma (\R^k)$ for any $k/2>\gamma>\gamma(k):=\frac{k(4k-1)}{14k-3}$. Moreover, the solution satisfies 
\[
\|u(T)\|_{H^\gamma(\R^k)} \leq C (1+T)^{ \frac{(4k-1)(k-2\gamma)}{2((14k-3)\gamma-k(4k-1))} + },
\] 
for $|T| \rightarrow \infty$, where the constant $C$ depends only on $\|u_0\|_{H^\gamma(\R^k)}$.
\end{theorem}
The proof of this theorem is based on the $I$-method similar to \cite{CoKeStaTaTao-almost-conservation} (see also \cite{Guo}). We shall consider a modified $I$-operator and show a suitable ``almost conservation law'' for the higher-order Schr\"odinger equation. The global well-posedness then follows by a usual scheme as in \cite{CoKeStaTaTao-almost-conservation}.  \newline 
\indent This paper is organized as follows. In Section $\ref{section preliminaries}$, we recall some linear and bilinear estimates for the higher-order Schr\"odinger equation, and also a modified $I$-operator together with its basic properties. We will show in Section $\ref{section almost conservation law}$ an almost conservation law and a modified local well-posed result. The proof of Theorem $\ref{theorem global existence}$ is proved in Section $\ref{section global proof}$. Throughout this paper, we shall use $A \lesssim B$ to denote an estimate of the form $A \leq C B$ for some absolute constant $C$. The notation $A \sim B$ means that $A \lesssim B$ and $B \lesssim A$. We write $A \ll B$ to denote $A \leq c B$ for some small constant $c>0$. We also use the Japanese bracket $\scal{a}:=\sqrt{1+|a|^2} \sim 1+|a|$ and $a\pm:= a\pm \ep$ with some universal constant $0<\ep \ll 1$. 
\section{Preliminaries} \label{section preliminaries}
\subsection{Littlewood-Paley decomposition}
Let $\varphi$ be a smooth, real-valued, radial function in $\R^k$ such that $\varphi(\xi)=1$ for $|\xi|\leq 1$ and $\varphi(\xi)=0$ for $|\xi|\geq 2$. 
Let $M=2^k, k \in \Z$. We denote the Littlewood-Paley operators by 
\begin{align*}
\widehat{P_{\leq M} f} (\xi) &:= \varphi(M^{-1}\xi) \hat{f}(\xi), \\
\widehat{P_{>M} f}(\xi) &:= (1-\varphi(M^{-1}\xi)) \hat{f}(\xi), \\
\widehat{P_M f}(\xi) &:= (\varphi(M^{-1}\xi) - \varphi(2M^{-1}\xi)) \hat{f}(\xi),
\end{align*}
where $\hat{\cdot}$ is the spatial Fourier transform. We similarly define 
\[
P_{<M}:= P_{\leq M} -P_M, \quad P_{\geq M} := P_{>M}+P_M,
\]
and for $M_1 \leq M_2$,
\[
P_{M_1<\cdot \leq M_2}:= P_{\leq M_2} - P_{\leq M_1} = \sum_{M_1<M\leq M_2} P_{M}.
\]
We have the following so called Bernstein's inequalities (see e.g. \cite[Chapter 2]{BCDfourier} or \cite[Appendix]{Tao}).
\begin{lem} \label{lem bernstein inequality}
Let $\gamma \geq 0$ and $1\leq p \leq q \leq \infty$. 
\begin{align*}
\|P_{\geq M} f\|_{L^p_x} &\lesssim M^{-\gamma}\| \Lambda^\gamma P_{\geq M} f\|_{L^p_x}, \\
\|P_{\leq M} \Lambda^\gamma f\|_{L^p_x} &\lesssim M^\gamma \|P_{\leq M} f\|_{L^p_x}, \\
\|P_M \Lambda^{\pm \gamma} f \|_{L^p_x} &\sim M^{\pm \gamma} \|P_M f\|_{L^p_x}, \\
\|P_{\leq M} f\|_{L^q_x} &\lesssim M^{k/p-k/q} \|P_{\leq M} f\|_{L^p_x}, \\
\|P_M f\|_{L^q_x} &\lesssim M^{k/p-k/q} \|P_M f\|_{L^p_x}.
\end{align*}
\end{lem}

\subsection{Norms and Strichartz estimates}
Let $\gamma, b\in \R$. The Bourgain space $X_{\tau=|\xi|^k}^{\gamma,b}$ is the closure of space-time Schwartz space $\mathscr{S}_{t,x}$ under the norm
\[
\|u\|_{X_{\tau=|\xi|^k}^{\gamma,b}}:= \|\scal{\xi}^\gamma \scal{\tau-|\xi|^k}^b \tilde{u}\|_{L^2_\tau L^2_\xi},
\] 
where $\tilde{\cdot}$ is the space-time Fourier transform,i.e. 
\[
\tilde{u}(\tau, \xi):= \iint e^{-i(t\tau+x\xi)} u(t,x)dt dx.
\]
We shall use $X^{\gamma,b}$ instead of $X^{\gamma,b}_{\tau=|\xi|^k}$ when there is no confusion. We recall a following special property of $X^{\gamma,b}$ space (see e.g. \cite[Lemma 2.9]{Tao}).
\begin{lem} \label{lem special property}
Let $\gamma, \gamma_1, \gamma_2 \in \R$ and $Y$ be a Banach space of functions on $\R \times \R^k$. If
\[
\|e^{it\tau} e^{it\Lambda^k} f\|_Y \lesssim \|f\|_{H^\gamma_x},
\]
for all $f \in H^\gamma_x$ and all $\tau \in \R$, then
\[
\|u\|_Y \lesssim \|u\|_{X^{\gamma,1/2+}},
\]
for all $u \in \mathscr{S}_{t,x}$. Moreover, if
\[
\|[e^{it\tau} e^{it\Lambda^k} f_1] [e^{it\zeta} e^{it\Lambda^k} f_2]\|_Y \lesssim \|f_1\|_{H^{\gamma_1}} \|f_2\|_{H^{\gamma_2}_x},
\]
for all $f_1 \in H^{\gamma_1}_x, f_2 \in H^{\gamma_2}_x$ and all $\tau, \zeta \in \R$, then 
\[
\|u_1 u_2\|_Y \lesssim \|u_1\|_{X^{\gamma_1,1/2+}} \|u_2\|_{X^{\gamma_2,1/2+}},
\]
for all $u_1, u_2 \in \mathscr{S}_{t,x}$.
\end{lem}
\indent Throughout this paper, a pair $(p,q)$ is called admissible in $\R^k$ if 
\begin{align}
(p,q)\in [2,\infty]^2, \quad (q,p)\ne (2,\infty), \quad \frac{1}{p}+\frac{1}{q}=\frac{1}{2}. \label{admissible} 
\end{align}
We recall the following Strichartz estimate (see e.g. \cite{Dinh}, \cite{Pausader}).
\begin{prop} \label{prop strichartz}
Let $k \geq 2, k \in \Z$. Suppose that $u$ is a solution to 
\[
i\partial_t u(t,x) + \Lambda^k u(t,x) = F(t,x), \quad u(0,x)=u_0(x), \quad (t,x)\in \R \times \R^k.
\]
Then for all $(p,q)$ and $(a,b)$ admissible pairs, 
\begin{align}
\|u\|_{L^p_t L^q_x} \lesssim \|u_0\|_{L^2_x} + \|F\|_{L^{a'}_t L^{b'}_x}. \nonumber
\end{align}
Here $(a,a')$ and $(b,b')$ are H\"older exponents.
\end{prop}
A direct consequence of Lemma $\ref{lem special property}$ and Proposition $\ref{prop strichartz}$ is the following linear estimate in $X^{\gamma,b}$ space.
\begin{coro}\label{coro linear Xsb}
Let $(p,q)$ be an admissible pair. Then 
\begin{align}
\|u\|_{L^p_t L^q_x} \lesssim \|u\|_{X^{0,1/2+}},  \label{strichartz estimate}
\end{align}
for all $u \in \mathscr{S}_{t,x}$.
\end{coro}
We also have the following bilinear estimate in $\R^k$.
\begin{prop} \label{prop bilinear}
Let $k\geq 2, k \in \Z$ and $M_1, M_2  \in 2^\Z$ be such that $M_1\leq M_2$. Then
\begin{align}
\|[e^{it\Lambda^k}P_{M_1} u_0] [e^{it\Lambda^k} P_{M_2} v_0]\|_{L^2_tL^2_x} \lesssim (M_1/M_2)^{(k-1)/2} \|u_0\|_{L^2_x} \|v_0\|_{L^2_x}. \nonumber
\end{align}
\end{prop}
\begin{proof}
We refer the reader to \cite{Bourgain-98} for the standard case $k=2$. The proof for $k>2$ is treated similarly. For $M_1 \sim M_2$, the result follows easily from the Strichartz estimate,
\begin{align*}
\|[e^{it\Lambda^k} P_{M_1} u_0][e^{it\Lambda^k} P_{M_2} v_0]\|_{L^2_t L^2_x} \leq \|e^{it\Lambda^k} P_{M_1} u_0\|_{L^4_t L^4_x} \|e^{it\Lambda^k} P_{M_2} v_0\|_{L^4_t L^4_x} \lesssim \|u_0\|_{L^2_x} \|v_0\|_{L^2_x}. 
\end{align*}
Note that $(4,4)$ is an admissible pair. Let us consider the case $M_1 \ll M_2$. By duality, it suffices to prove 
\begin{multline}
\Big|\iint_{\R^k \times \R^k} G(-|\xi|^k - |\eta|^k, \xi+\eta) \widehat{P_{M_1} u}_0(\xi) \widehat{P_{M_2} v}_0(\eta) d\xi d\eta\Big| \\
\lesssim (M_1/M_2)^{(k-1)/2} \|G\|_{L^2_\tau L^2_\xi} \|\hat{u}_0\|_{L^2_\xi} \|\hat{v}_0\|_{L^2_\xi}. \label{equivalent bilinear}
\end{multline}
By renaming the components, we can assume that $|\xi|\sim |\xi_1| \sim M_1$ and $|\eta|\sim |\eta_1| \sim M_2$, where $\xi=(\xi_1, \underline{\xi}), \eta=(\eta_1, \underline{\eta})$ with $\underline{\xi}, \underline{\eta} \in \R^{k-1}$. We make a change of variables $\tau=-|\xi|^k-|\eta|^k, \vartheta= \xi+\eta$ and $d\tau d\vartheta = J d\xi_1 d\eta$. An easy computation shows that $J=|k(|\eta|^{k-2}\eta_1 - |\xi|^{k-2}\xi_1)| \sim |\eta|^{k-1} \sim M_2^{k-1}$. The Cauchy-Schwarz inequality with the fact that $|\underline{\xi}| \lesssim M$ then yields
\begin{align*}
\text{LHS}(\ref{equivalent bilinear}) &= \Big| \iiint_{\R \times \R^k \times \R^{k-1}} G(\tau, \vartheta) \widehat{P_{M_1} u}_0(\xi) \widehat{P_{M_2} v}_0(\eta) J^{-1} d\tau d\vartheta d\underline{\xi} \Big| \\
& \leq \|G\|_{L^2_\tau L^2_\xi} \int_{\R^{k-1}} \Big(\iint_{\R \times \R^k} |\widehat{P_{M_1} u}_0(\xi)|^2 |\widehat{P_{M_2} v}_0(\eta)|^2 J^{-2} d\tau d\vartheta \Big)^{1/2} d\underline{\xi} \\
&\leq \|G\|_{L^2_\tau L^2_\xi} M_1^{(k-1)/2} \Big( \iiint_{\R \times \R^k \times \R^{k-1}} |\widehat{P_{M_1} u}_0(\xi)|^2 |\widehat{P_{M_2} v}_0(\eta)|^2 J^{-2} d\tau d\vartheta d\underline{\xi} \Big)^{1/2} \\
&\leq \|G\|_{L^2_\tau L^2_\xi} M_1^{(k-1)/2} \Big( \iiint_{\R \times \R^k \times \R^{k-1}} |\widehat{P_{M_1} u}_0(\xi)|^2 |\widehat{P_{M_2} v}_0(\eta)|^2 J^{-1} d\xi d\eta \Big)^{1/2} \\
&\lesssim \|G\|_{L^2_\tau L^2_\xi} (M_1/M_2)^{(k-1)/2} \|\widehat{P_{M_1} u}_0\|_{L^2_\xi} \|\widehat{P_{M_2} v}_0\|_{L^2_\xi}.
\end{align*}
This proves $(\ref{equivalent bilinear})$, and the proof is complete.
\end{proof}
The following result is another application of Lemma $\ref{lem special property}$ and Proposition $\ref{prop bilinear}$.
\begin{coro} \label{coro bilinear Xsb}
Let $k\geq 2, k\in \Z$ and $u_1, u_2 \in X^{0,1/2+}$ be supported on spatial frequencies $|\xi| \sim M_1, M_2$ respectively. Then for $M_1 \leq M_2$, 
\begin{align}
\|u_1 u_2\|_{L^2_t L^2_x} \lesssim (M_1/M_2)^{(k-1)/2} \|u_1\|_{X^{0,1/2+}} \|u_2\|_{X^{0,1/2+}}. \label{bilinear estimate}
\end{align}
A similar estimate holds for $\overline{u}_1 u_2$ or $u_1 \overline{u}_2$. 
\end{coro}
\subsection{$I$-operator}
For $0<\gamma <k/2$ and $N\gg 1$, we define the Fourier multiplier $I_N$ by
\begin{align}
\widehat{I_N f} (\xi):= m_N(\xi) \hat{f}(\xi), \label{I operator}
\end{align}
where $m$ is a smooth, radially symmetric, non-increasing function such that 
\begin{align}
m_N(\xi):= \left\{ 
\begin{array}{c l}
1 & \text{if } |\xi| \leq N, \\
(N^{-1}|\xi|)^{\gamma-2} & \text{if } |\xi|\geq 2N.
\end{array}
\right. \label{define m}
\end{align}
For simplicity, we shall drop the $N$ from the notation and write $I$ and $m$ instead of $I_N$ and $m_N$. The operator $I$ is the identity on low frequencies $|\xi|\leq N$ and behaves like a fractional integral operator of order $k/2-\gamma$ on high frequencies $|\xi|\geq 2N$. We recall some basic properties of the $I$-operator in the following lemma.
\begin{lem} \label{lem property I operator}
Let $q \in (1,\infty)$ and $\gamma \in (0,k/2)$. Then
\begin{align}
\|I f\|_{L^q_x} & \lesssim \|f\|_{L^q_x}, \label{Lp bound I operator} \\
\|f\|_{H^\gamma_x} \lesssim \|If\|_{H^{k/2}_x} &\lesssim N^{k/2-\gamma} \|f\|_{H^\gamma_x}. \label{sobolev bound I operator}
\end{align}
\end{lem}
\begin{proof}
The estimate $(\ref{Lp bound I operator})$ follows from the fact that $m$ satisfies the H\"ormander multiplier condition. For $(\ref{sobolev bound I operator})$, we proceed as follows. 
\begin{align*}
\|f\|^2_{H^\gamma_x} &\lesssim \int_{|\xi| \leq N} \scal{\xi}^{2\gamma} |\widehat{I f}(\xi)|^2 d\xi + \int_{|\xi|\geq 2N} \scal{\xi}^{2\gamma} (N^{-1}|\xi|)^{2(k/2-\gamma)} |\widehat{I f}(\xi)|^2 d\xi \\
& \lesssim \int_{|\xi| \leq N} \scal{\xi}^k |\widehat{I f}(\xi)|^2 d\xi + \int_{|\xi|\geq 2N} \scal{\xi}^k |\widehat{I f}(\xi)|^2 d\xi \lesssim \|If\|^2_{H^{k/2}_x}.
\end{align*}
This gives the first estimate in $(\ref{sobolev bound I operator})$. Similarly,
\begin{align*}
\|I f\|^2_{H^{k/2}_x} &\lesssim \int_{|\xi| \leq N} \scal{\xi}^k |\hat{f}(\xi)|^2 d\xi + \int_{|\xi|\geq 2N} \scal{\xi}^k (N^{-1}|\xi|)^{2(\gamma-k/2)} |\hat{f}(\xi)|^2d\xi \\
&\lesssim \int_{|\xi|\leq N} \scal{\xi}^{2(k/2-\gamma)} \scal{\xi}^{2\gamma} |\hat{f}(\xi)|^2 d\xi + \int_{|\xi|\geq 2N} N^{2(k/2-\gamma)} \scal{\xi}^{2\gamma} |\hat{f}(\xi)|^2 d\xi \\
& \lesssim N^{2(k/2-\gamma)} \Big(\int_{|\xi|\leq N} \scal{\xi}^{2\gamma}|\hat{f}(\xi)|^2 d\xi + \int_{|\xi|\geq 2N}\scal{\xi}^{2\gamma}|\hat{f}(\xi)|^2 d\xi \Big) \lesssim N^{2(k/2-\gamma)} \|f\|^2_{H^{\gamma}_x}.
\end{align*}
The proof is complete.
\end{proof}
\section{Almost conservation law} \label{section almost conservation law}
As mentioned in the introduction, the equation $(\Ntext)$ is locally well-posed in $H^\gamma$ for any $\gamma>0$. Moreover, the time of existence depends only on the $H^\gamma_x$-norm of the initial data. Thus, the global well-posedness will follows from a global $L^\infty_t H^\gamma_x$ bound of the solution by the usual iterative argument. For $H^\gamma$ solution with $\gamma\geq k/2$, one can obtain easily the $L^\infty_tH^\gamma_x$ bound of solution using the persistence of regularity and the conserved quantities of mass and energy. But it is not the case for $H^\gamma$ solution with $\gamma<k/2$ since the energy is no longer conserved. However, it follows from $(\ref{sobolev bound I operator})$ that the $H^\gamma_x$-norm of the solution $u$ can be controlled by the $H^{k/2}_x$-norm of $Iu$. It leads to consider the following modified energy functional
\begin{align}
E(Iu(t)):= \frac{1}{2} \|Iu(t)\|^2_{\dot{H}^{k/2}_x} + \frac{1}{4}\|Iu(t)\|^4_{L^4_x}. \label{modified energy functional}
\end{align}
Since $Iu$ is not a solution to $(\Ntext)$, we can expect an energy increment. We have the following ``almost conservation law''. 
\begin{prop} \label{prop almost conservation law}
Let $k\geq 3, k \in \Z$. Given $k/2>\gamma>\gamma(k):=\frac{k(4k-1)}{14k-3}$, $N\gg 1$, and initial data $u_0 \in C^\infty(\R^k)$ with $E(Iu_0)\leq 1$, then there exists a $\delta=\delta(\|u_0\|_{L^2_x})>0$ so that the solution $u \in C([0,\delta],H^\gamma(\R^k))$ of $(\Nemph)$ satisfies
\begin{align}
E(Iu(t)) = E(Iu_0) + O(N^{-\gamma_0(k)+}), \label{almost conservation law}
\end{align}
where $\gamma_0(k):= \frac{k(6k-1)}{8k-2}$ for all $t\in [0,\delta]$.
\end{prop}
\begin{rem} \label{rem almost conservation law}
This proposition tells us that the modified energy $E(Iu(t))$ decays with respect to the parameter $N$. We will see in Section $\ref{section global proof}$ that if we can replace the increment $N^{-\gamma_0(k)+}$ in the right hand side of $(\ref{almost conservation law})$ with $N^{-\gamma_1(k)+}$ for some $\gamma_1(k)>\gamma_0(k)$, then the global existence can be improved for all $\gamma>\frac{k^2}{2(k+\gamma_1(k))}$. In particular, if $\gamma_1(k)=\infty$, then $E(Iu(t))$ is conserved, and the global well-posedness holds for all $\gamma>0$. 
\end{rem}
In order to prove Proposition $\ref{prop almost conservation law}$, we recall the following interpolation result (see \cite[Lemma 12.1]{CoKeStaTaTao-multilinear}). Let $\eta$ be a smooth, radial, decreasing function which equals 1 for $|\xi| \leq 1$ and equals $|\xi|^{-1}$ for $|\xi|\geq 2$. For $N\geq 1$ and $\alpha \in \R$, we define the spatial Fourier multiplier $J^\alpha_N$ by
\begin{align}
\widehat{J^\alpha_N f}(\xi):= (\eta(N^{-1}\xi))^\alpha \hat{f}(\xi). \label{J operator}
\end{align}
The operator $J^\alpha_N$ is a smoothing operator of order $\alpha$, and it is the identity on the low frequencies $|\xi|\leq N$. 
\begin{lem}[Interpolation \cite{CoKeStaTaTao-multilinear}] \label{lem interpolation}
Let $\alpha_0>0$ and $ n\geq 1$. Suppose that $Z, X_1,..., X_n$ are translation invariant Banach spaces and $T$ is a translation invariant $n$-linear operator such that 
\[
\|J^\alpha_1 T(u_1,...,u_n)\|_Z \lesssim \prod_{i=1}^n \|J^\alpha_1 u_i\|_{X_i},
\]
for all $u_1,...,u_n$ and all $0 \leq \alpha \leq \alpha_0$. Then one has
\[
\|J^\alpha_N T(u_1,...,u_n)\|_Z \lesssim \prod_{i=1}^n \|J^\alpha_N u_i\|_{X_i},
\]
for all $u_1,...,u_n$, all $0\leq \alpha\leq \alpha_0$, and $N\geq 1$, with the implicit constant independent of $N$. 
\end{lem}
Using this interpolation lemma, we are able to prove the following modified version of the usual local well-posed result. 
\begin{prop} \label{prop modified local well-posedness}
Let\footnote{see Theorem \ref{theorem global existence} for the definition of $\gamma(k)$.} $\gamma \in (\gamma(k),k/2)$ and $u_0 \in H^\gamma(\R^k)$ be such that $E(Iu_0) \leq 1$. Then there is a constant $\delta=\delta(\|u_0\|_{L^2_x})$ so that the solution $u$ to $(\Nemph)$ satisfies
\begin{align}
\|Iu\|_{X^{k/2,1/2+}_\delta} \lesssim 1. \label{modified local well-posedness}
\end{align} 
Here $X^{\gamma,b}_\delta$ is the space of restrictions of elements of $X^{\gamma,b}$ endowed with the norm
\begin{align}
\|u\|_{X^{\gamma,b}_\delta}:= \inf \{ \|w\|_{X^{\gamma,b}} \ | \ w_{\vert [0,\delta]\times \R^k} = u\}. \label{restriction Xsb norm}
\end{align} 
\end{prop}
\begin{proof}
We recall the following estimates involving the $X^{\gamma,b}$ spaces which are proved in the Appendix. Let $\gamma \in \R$ and $\psi \in C^\infty_0(\R)$ be such that $\psi(t) =1$ for $t \in [-1,1]$.  One has
\begin{align}
\|\psi(t) e^{it\Lambda^k} u_0\|_{X^{\gamma,b}} &\lesssim \|u_0\|_{H^\gamma_x}, \label{homogeneous estimate} \\
\Big\|\psi_\delta(t)\int_0^t e^{i(t-s)\Lambda^k} F(s) ds\Big\|_{X^{\gamma,b}} &\lesssim \delta^{1-b-b'} \|F\|_{X^{\gamma,-b'}}, \label{inhomogeneous estimate}
\end{align}
where $\psi_\delta(t):=\psi(\delta^{-1}t)$ provided $0 <\delta \leq 1$ and
\begin{align}
0<b'<1/2<b, \quad b+b' <1. \label{inhomogeneous estimate condition} 
\end{align}
Note that the implicit constants are independent of $\delta$. This implies for $0<\delta \leq 1$ and $b,b'$ as in $(\ref{inhomogeneous estimate condition})$ that
\begin{align}
\|e^{it\Lambda^k} u_0\|_{X^{\gamma,b}_\delta} &\lesssim \|u_0\|_{H^\gamma_x}, \label{restrition homogeneous estimate}\\
\Big\|\int_0^t e^{i(t-s)\Lambda^k} F(s) ds \Big\|_{X^{\gamma,b}_\delta} &\lesssim \delta^{1-b-b'}\|F\|_{X^{\gamma,-b'}_\delta}. \label{restriction inhomogeneous estimate} 
\end{align}
By the Duhamel principle, we have
\[
\|I u\|_{X^{k/2,b}_\delta} = \Big\| e^{it\Lambda^k} I u_0 + \int_0^t e^{it\Lambda^k} I(|u|^2u)(s)ds \Big\|_{X^{k/2,b}_\delta} \lesssim \|Iu_0\|_{H^{k/2}_x} + \delta^{1-b-b'} \|I(|u|^2u)\|_{X^{k/2,-b'}_\delta}. 
\]
By the definition of restriction norm $(\ref{restriction Xsb norm})$, 
\[
\|Iu\|_{X^{k/2,b}_\delta} \lesssim \|I u_0\|_{H^{k/2}_x} + \delta^{1-b-b'} \|I(|w|^2 w)\|_{X^{k/2,-b'}},
\]
where $w$ agrees with $u$ on $[0,\delta]\times \R^k$ and 
\[
\|I u\|_{X^{k/2,b}_\delta} \sim \|I w\|_{X^{k/2,b}}.
\]
Let us assume for the moment that 
\begin{align}
\|I(|w|^2 w)\|_{X^{k/2,-b'}} \lesssim \|I w\|^3_{X^{k/2,b}}. \label{trilinear estimate}
\end{align}
This implies that
\[
\|I u\|_{X^{k/2,b}_\delta} \lesssim \|Iu_0\|_{H^{k/2}_x} + \delta^{1-b-b'} \|Iu\|^3_{X^{k/2,b}_\delta}.
\]
Note that 
\[
\|Iu_0\|_{H^{k/2}_x} \sim \|Iu_0\|_{\dot{H}^{k/2}_x} + \|Iu_0\|_{L^2_x} \leq 1 + \|u_0\|_{L^2_x}.
\]
As $\|Iu\|_{X^{k/2,b}_\delta}$ is continuous in the $\delta$ variable, the bootstrap argument (see e.g. \cite[Section 1.3]{Tao}) yields
\[
\|Iu\|_{X^{k/2,b}_\delta} \lesssim 1.
\]
This proves $(\ref{modified local well-posedness})$. It remains to show $(\ref{trilinear estimate})$. We will take the advantage of interpolation Lemma $\ref{lem interpolation}$. Note that the $I$-operator defined in $(\ref{I operator})$ is equal to $J^\alpha_N$ defined in $(\ref{J operator})$ with $\alpha=k/2-\gamma$. Thus, by Lemma $\ref{lem interpolation}$, $(\ref{trilinear estimate})$ is proved once there is $\alpha_0>0$ so that 
\[
\|J^\alpha_1(|w|^2w)\|_{X^{k/2,-b'}} \lesssim \|J^\alpha_1 w\|^3_{X^{k/2,b}}, 
\]
for all $0\leq \alpha \leq \alpha_0$. Splitting $w$ to low and high frequency parts $|\xi|\lesssim 1$ and $|\xi|\gg 1$ respectively and using definition of $J^\alpha_1$, it suffices to show
\begin{align}
\||w|^2 w\|_{X^{\gamma,-b'}} \lesssim \|w\|^3_{X^{\gamma,b}}, \label{interpolation application 1}
\end{align}
for all $\gamma \in [\gamma(k),k/2]$. By duality, a Leibniz rule, $(\ref{interpolation application 1})$ follows from
\begin{align}
\Big| \iint_{\R \times \R^k} (\scal{\Lambda}^\gamma w_1) \overline{w}_2 w_3 w_4 dt dx \Big| \lesssim \|w_1\|_{X^{\gamma,b}} \|w_2\|_{X^{\gamma,b}} \|w_3\|_{X^{\gamma,b}} \|w_4\|_{X^{0,b'}}. \label{duality}
\end{align}
Note that the last term should be precise as $\|w_4\|_{X^{0,b'}_{\tau=-|\xi|^k}}$ but it does not effect our estimate. Using H\"older's inequality, we can bound the left hand side of $(\ref{duality})$ as
\[
\text{LHS}(\ref{duality}) \leq \|\scal{\Lambda}^\gamma w_1 \|_{L^4_t L^4_x} \|w_2\|_{L^4_tL^4_x} \|w_3\|_{L^6_t L^6_x} \|w_4\|_{L^3_t L^3_x}.
\]
Since $(4,4)$ is an admissible pair, Corollary $\ref{coro linear Xsb}$ gives
\[
\|\scal{\Lambda}^\gamma w_1\|_{L^4_t L^4_x} \lesssim \|w_1\|_{X^{\gamma,b}}, \quad \|w_2\|_{L^4_tL^4_x}\lesssim \|w_2\|_{X^{0,b}} \leq \|w_2\|_{X^{\gamma,b}}.
\]
Similarly, Sobolev embedding and Corollary $\ref{coro linear Xsb}$ yield
\[
\|w_3\|_{L^6_t L^6_x} \lesssim \|\scal{\Lambda}^{k/6} w_3\|_{L^6_t L^3_x} \lesssim \|w_3\|_{X^{k/6},b} \leq \|w_3\|_{X^{\gamma,b}}.
\]
The last estimate comes from the fact that $\gamma>\gamma(k)>k/6$. Finally, we interpolate between $\|w_4\|_{L^2_tL^2_x}= \|w_4\|_{X^{0,0}}$ and $\|w_4\|_{L^4_t L^4_x} \lesssim \|w_4\|_{X^{0,1/2+}}$ to get
\[
\|w_4\|_{L^3_t L^3_x}\lesssim \|w_4\|_{X^{0,b'}}. 
\]
Combing these estimates, we have $(\ref{duality})$. The proof of Proposition $\ref{prop modified local well-posedness}$ is now complete.
\end{proof}
We are now able to prove the almost conservation law.
\paragraph{Proof of Proposition $\ref{prop almost conservation law}$.}
By the assumption $E(Iu_0)\leq 1$, Proposition $\ref{prop modified local well-posedness}$ shows that there exists $\delta=\delta(\|u_0\|_{L^2_x})$ such that the solution $u$ to $(\Ntext)$ satisfies $(\ref{modified local well-posedness})$. We firstly note that the usual energy satisfies
\begin{align*}
\frac{d}{dt} E(u(t)) &= \re \int_{\R^k} \overline{\partial_t u(t,x)} (|u(t,x)|^2 u(t,x) + \Lambda^k u(t,x)) dx\\
&= \re \int_{\R^k} \overline{\partial_t u(t,x)} (|u(t,x)|^2 u(t,x) + \Lambda^k u(t,x) + i \partial_t u(t,x)) dx =0.
\end{align*}
Similarly, we have
\begin{align*}
\frac{d}{dt} E(Iu(t)) &= \re \int_{\R^k} \overline{I \partial_t u(t,x)} (|Iu(t,x)|^2 Iu(t,x) + \Lambda^k Iu(t,x) + i \partial_t Iu(t,x)) dx \\
&= \re \int_{\R^k} \overline{I \partial_t u(t,x)} (|Iu(t,x)|^2 Iu(t,x) - I(|u(t,x)|^2u(t,x))) dx. 
\end{align*}
Here the second line follows by applying $I$ to both sides of $(\Ntext)$. Integrating in time and applying the Parseval formula, we obtain
\[
E(Iu(t))-E(Iu_0)= \re \int_0^\delta \int_{\sum_{j=1}^{4} \xi_j=0} \Big(1-\frac{m(\xi_2+\xi_3+\xi_4)}{m(\xi_2)m(\xi_3)m(\xi_4)} \Big) \widehat{\overline{I \partial_t u}} (\xi_1) \widehat{I u}(\xi_2) \widehat{\overline{I u}}(\xi_3) \widehat{I u}(\xi_4) dt.
\]
Here $\int_{\sum_{j=1}^4 \xi_j=0}$ denotes the integration with respect to the hyperplane's measure $\delta_0(\xi_1+...+\xi_4) d\xi_1...d\xi_4$. Using that $i I \partial_t u = - \Lambda^k I u - I(|u|^2u)$, we have
\[
|E(Iu(t))-E(Iu_0)| \leq \text{Term}_1 + \text{Term}_2,
\]
where 
\[
\text{Term}_1= \Big| \int_0^\delta \int_{\sum_{j=1}^{4}\xi_j=0} \mu(\xi_2,\xi_3,\xi_4) \widehat{\overline{\Lambda^k I u}}(\xi_1) \widehat{I u}(\xi_2) \widehat{\overline{I u}}(\xi_3) \widehat{I u}(\xi_4) dt \Big|,
\]
and 
\[
\text{Term}_2= \Big| \int_0^\delta \int_{\sum_{j=1}^{4}\xi_j=0} \mu(\xi_2,\xi_3,\xi_4) \widehat{\overline{I (|u|^2u)}}(\xi_1) \widehat{I u}(\xi_2) \widehat{\overline{I u}}(\xi_3) \widehat{I u}(\xi_4) dt \Big|,
\]
with 
\[
\mu(\xi_2,\xi_3,\xi_4):= 1-\frac{m(\xi_2+\xi_3+\xi_4)}{m(\xi_2)m(\xi_3)m(\xi_4)}.
\]
Our purpose is to prove
\begin{align*}
\text{Term}_1 +\text{Term}_2 \lesssim N^{-\gamma_0(k)+}. 
\end{align*}
\indent Let us consider the first term ($\text{Term}_1$). To do so, we decompose $u= \sum_{M\geq 1} P_M u=: \sum_{M\geq 1} u_M$ with the convention $P_1 u:=P_{\leq 1} u$ and write $\text{Term}_1$ as a sum over all dyadic pieces. By the symmetry of $\mu$ in $\xi_2, \xi_3, \xi_4$ and the fact that the bilinear estimate $(\ref{bilinear estimate})$ allows complex conjugations on either factors, we may assume that $M_2 \geq M_3 \geq M_4$. Thus,
\[
\text{Term}_1 \lesssim \sum_{M_1,M_2,M_3,M_4 \geq 1 \atop M_2 \geq M_3 \geq M_4} A(M_1, M_2, M_3,M_4),
\]
where 
\[
A(M_1,M_2, M_3, M_4):= \Big|\int_0^\delta \int_{\sum_{j=1}^4 \xi_j=0} \mu(\xi_2, \xi_3, \xi_4) \widehat{\overline{\Lambda^k I u_{M_1}}}(\xi_1) \widehat{I u_{M_2}}(\xi_2) \widehat{\overline{I u_{M_3}}}(\xi_3) \widehat{I u_{M_4}}(\xi_4) dt  \Big|.
\]
For simplifying the notation, we will drop the dependence of $M_1, M_2, M_3, M_4$ and write $A$ instead of $A(M_1, M_2, M_3, M_4)$. In order to have $\text{Term}_1 \lesssim N^{-\gamma_0(k)+}$, it suffices to prove
\begin{align}
A \lesssim N^{-\gamma_0(k)+} M_2^{0-}. \label{estimation of A}
\end{align}
To show $(\ref{estimation of A})$, we will break the frequency interactions into three cases due to the comparison of $N$ with $M_j$. It is worth to notice that $M_1 \lesssim M_2$ due to the fact that $\sum_{j=1}^4 \xi_j=0$. \newline
\textbf{Case 1.} $N \gg M_2$. In this case, we have $|\xi_2|, |\xi_3|, |\xi_4|\ll N$ and $|\xi_2+\xi_3+\xi_4| \leq N$, hence 
\[
m(\xi_2+\xi_3+\xi_4)=m(\xi_2)=m(\xi_3)=m(\xi_4)=1 \text{ and } \mu(\xi_2,\xi_3,\xi_4)=0.
\]
Thus $(\ref{estimation of A})$ holds trivially. \newline
\textbf{Case 2.} $M_2 \gtrsim N \gg M_3 \geq M_4$. Since $\sum_{j=1}^4 \xi_j=0$, we get $M_1 \sim M_2$. We also have from the mean value theorem that
\[
|\mu(\xi_2,\xi_3,\xi_4)| = \Big|1-\frac{m(\xi_2+\xi_3+\xi_4)}{m(\xi_2)} \Big| \lesssim \frac{|\nabla m(\xi_2)\cdot(\xi_3+\xi_4)|}{m(\xi_2)} \lesssim \frac{M_3}{M_2}.
\]
The pointwise bound, H\"older's inequality, Plancherel theorem and bilinear estimate $(\ref{bilinear estimate})$ yield
\begin{align}
A &\lesssim \frac{M_3}{M_2} \|\Lambda^k I u_{M_1} Iu_{M_3}\|_{L^2_t L^2_x} \|Iu_{M_2} Iu_{M_4}\|_{L^2_t L^2_x} \nonumber \\
& \lesssim \frac{M_3}{M_2} \Big(\frac{M_3}{M_1}\Big)^{(k-1)/2} \Big(\frac{M_4}{M_2}\Big)^{(k-1)/2} M_1^k \prod_{j=1}^{4}\|Iu_{M_j}\|_{X^{0,1/2+}} \nonumber \\
&\lesssim \frac{M_3}{M_2} \Big(\frac{M_3}{M_1}\Big)^{(k-1)/2} \Big(\frac{M_4}{M_2}\Big)^{(k-1)/2} \frac{M_1^{k/2}}{M_2^{k/2} \scal{M_3}^{k/2} \scal{M_4}^{k/2}} \prod_{j=1}^4 \|I u_{M_j}\|_{X^{k/2,1/2+}} \nonumber \\
&=\Big(\frac{M_3}{N}\Big)^{1/2} \Big(\frac{M_1}{M_2}\Big)^{1/2} \Big(\frac{N}{M_2}\Big)^{k-}N^{-(k-1/2)+} M_2^{0-}  \prod_{j=1}^4 \|I u_{M_j}\|_{X^{k/2,1/2+}} \nonumber \\
&\lesssim N^{-(k-1/2)+} M_2^{0-} \prod_{j=1}^4 \|I u_{M_j}\|_{X^{k/2,1/2+}}. \label{estimate case 2}
\end{align}
Using $(\ref{modified local well-posedness})$ and the fact that $\gamma_0(k)<k-1/2$ for $k\geq 3, k \in \Z$, we have $(\ref{estimation of A})$. \newline
\textbf{Case 3.} $M_2 \geq M_3 \gtrsim N$. In this case, we simply bound
\[
|\mu(\xi_2, \xi_3, \xi_4)| \lesssim \frac{m(\xi_1)}{m(\xi_2)m(\xi_3)m(\xi_4)}. 
\]  
Here we use that $m(\xi_1)\gtrsim m(\xi_2)$ and $m(\xi_3)\leq m(\xi_4) \leq 1$ due to the fact that $M_1 \lesssim M_2$ and $M_3 \geq M_4$. \newline
\textbf{Subcase 3a.} $M_2 \gg M_3 \gtrsim N$. We see that $M_1 \sim M_2$ since $\sum_{j=1}^4 \xi_j=0$. The pointwise bound, H\"older's inequality, Plancherel theorem and bilinear estimate $(\ref{bilinear estimate})$ again give
\begin{align*}
A&  \lesssim \frac{m(M_1)}{m(M_2)m(M_3)m(M_4)} \|\overline{\Lambda^k I u_{M_1}} I u_{M_4}\|_{L^2_tL^2_x} \|Iu_{M_2} \overline{I u_{M_3}}\|_{L^2_t L^2_x} \\
&\lesssim \frac{m(M_1)}{m(M_2)m(M_3)m(M_4)} \Big(\frac{M_4}{M_1}\Big)^{(k-1)/2} \Big(\frac{M_3}{M_2}\Big)^{(k-1)/2} \frac{M_1^{k/2}}{M_2^{k/2} M_3^{k/2} \scal{M_4}^{k/2}} \prod_{j=1}^4 \|Iu_{M_j}\|_{X^{k/2,1/2+}}.
\end{align*}
Thanks to $(\ref{modified local well-posedness})$, we only need to show
\begin{align}
\frac{m(M_1)}{m(M_2)m(M_3)m(M_4)} \Big(\frac{M_4}{M_1}\Big)^{(k-1)/2} \Big(\frac{M_3}{M_2}\Big)^{(k-1)/2} \frac{M_1^{k/2}}{M_2^{k/2} M_3^{k/2} \scal{M_4}^{k/2}} \lesssim N^{-\gamma_0(k)+} M_2^{0-}. \label{case 3a}
\end{align}
Remark that the function $m(\lambda)\lambda^\alpha$ is increasing, and $m(\lambda)\scal{\lambda}^\alpha$ is bounded below for any $\alpha +\gamma - k/2>0$ due to
\[
(m(\lambda)\lambda^\alpha)' = \left\{ \begin{array}{c l}
\alpha \lambda^{\alpha-1} & \text{if } 1 \leq \lambda \leq N, \\
N^{k/2-\gamma}(\alpha+\gamma-k/2) \lambda^{\alpha+\gamma-k/2-1} & \text{if } \lambda \geq 2N.
\end{array} \right.
\]  
We shall shortly choose an appropriate value of $\alpha$, says $\alpha(k)$, so that 
\begin{align}
m(M_4) \scal{M_4}^{\alpha(k)} \gtrsim 1, \quad m(M_3) M_3^{\alpha(k)} \gtrsim m(N)N^{\alpha(k)} =N^{\alpha(k)}. \label{property of m}
\end{align}
Using that $m(M_1)\sim m(M_2)$, we have
\begin{align*}
\text{LHS}(\ref{case 3a}) &\lesssim \frac{M_3^{\alpha(k)-1/2} \scal{M_4}^{\alpha(k)-1/2} M_1^{1/2}}{m(M_3)M_3^{\alpha(k)} m(M_4)\scal{M_4}^{\alpha(k)} M_2^{k-1/2}} \\
& \lesssim \frac{1}{N^{\alpha(k)} M_2^{k-2\alpha(k)}} \Big(\frac{M_3}{M_2}\Big)^{\alpha(k)-1/2} \Big(\frac{\scal{M_4}}{M_2}\Big)^{\alpha(k)-1/2} \Big(\frac{M_1}{M_2}\Big)^{1/2} \\
&\lesssim N^{-(k-\alpha(k))+} M_2^{0-}. 
\end{align*}
Therefore, if we choose $\alpha(k)$ so that $\gamma_0(k) = k-\alpha(k)$ or $\alpha(k)= k-\gamma_0(k)=\frac{k(2k-1)}{8k-2}$, then we get $(\ref{estimation of A})$. Note that $\alpha(k)+\gamma(k)-k/2 \geq 0$ for $k\geq 3, k \in \Z$, hence $(\ref{property of m})$ holds. \newline
\textbf{Subcase 3b.} $M_2 \sim M_3 \gtrsim N$. In this case, we see that $M_1 \lesssim M_2$. Arguing as in Subcase 3a, we obtain
\begin{align*}
A&  \lesssim \frac{m(M_1)}{m(M_2)m(M_3)m(M_4)} \|\overline{\Lambda^k I u_{M_1}} I u_{M_2}\|_{L^2_tL^2_x} \|Iu_{M_3} \overline{I u_{M_4}}\|_{L^2_t L^2_x} \\
&\lesssim \frac{m(M_1)}{m(M_2)m(M_3)m(M_4)} \Big(\frac{M_1}{M_2}\Big)^{(k-1)/2} \Big(\frac{M_4}{M_3}\Big)^{(k-1)/2} \frac{\scal{M_1}^{k/2}}{M_2^{k/2} M_3^{k/2} \scal{M_4}^{k/2}} \prod_{j=1}^4 \|Iu_{M_j}\|_{X^{k/2,1/2+}}.
\end{align*}
As in Subcase 3a, our aim is to prove
\begin{align}
\frac{m(M_1)}{m(M_2)m(M_3)m(M_4)} \Big(\frac{M_1}{M_2}\Big)^{(k-1)/2} \Big(\frac{M_4}{M_3}\Big)^{(k-1)/2} \frac{\scal{M_1}^{k/2}}{M_2^{k/2} M_3^{k/2} \scal{M_4}^{k/2}} \lesssim N^{-\gamma_0(k)+} M_2^{0-}. \label{case 3b}
\end{align}
We use $(\ref{property of m})$ to get
\begin{align*}
\text{LHS}(\ref{case 3b}) &\lesssim \frac{m(M_1)}{m(M_2)m(M_3)m(M_4) \scal{M_4}^{1/2} M_3^{k-1/2}} \\
& \lesssim \frac{m(M_1) M_2^{\alpha(k)}  \scal{M_4}^{\alpha(k)-1/2}}{ m(M_2) M_2^{\alpha(k)} m(M_3) M_3^{\alpha(k)} m(M_4)\scal{M_4}^{\alpha(k)} M_3^{k-\alpha(k)-1/2}} \\
& \lesssim \frac{1}{N^{2\alpha(k)}} \Big(\frac{M_2}{M_3} \Big)^{\alpha(k)} \Big(\frac{\scal{M_4}}{M_3}\Big)^{\alpha(k)-1/2} \frac{1}{M_3^{k-3\alpha(k)}}  \\
&\lesssim N^{-(k-\alpha(k))+} M_2^{0-}.
\end{align*}
Choosing $\alpha(k)$ as in Subcase 3a, we get $(\ref{estimation of A})$. \newline
\indent We now consider the second term ($\text{Term}_2$). We again decompose $u$ in dyadic frequencies, $u=\sum_{M\geq 1} u_M$. By the symmetry, we can assume that $M_2 \geq M_3 \geq M_4$. We can assume further that $M_2 \gtrsim N$ since $\mu(\xi_2, \xi_3, \xi_4)$ vanishes otherwise. Thus, 
\[
\text{Term}_2 \lesssim \sum_{M_1, M_2, M_3, M_4 \geq 1 \atop M_2 \geq M_3 \geq M_4} B(M_1, M_2, M_3, M_4),
\]
where
\[
B(M_1, M_2, M_3,M_4) := \Big| \int_0^\delta \int_{\sum_{j=1}^4 \xi_j=0} \mu(\xi_2, \xi_3, \xi_4) \widehat{\overline{P_{M_1} I (|u|^2u)}}(\xi_1) \widehat{I u_{M_2}}(\xi_2) \widehat{\overline{I u_{M_3}}}(\xi_3) \widehat{I u_{M_4}}(\xi_4) dt\Big|.
\]
As for the $\text{Term}_1$, we will use the notation $B$ instead of $B(M_1, M_2, M_3, M_4)$. Using the trivial bound 
\[
|\mu(\xi_2, \xi_3, \xi_4)| \lesssim \frac{m(M_1)}{m(M_2) m(M_3) m(M_4)},
\]
H\"older's inequality and Plancherel theorem, we bound 
\[
B\lesssim \frac{m(M_1)}{m(M_2) m(M_3) m(M_4)} \| P_{M_1} I(|u|^2u)\|_{L^2_t L^2_x} \|I u_{M_2}\|_{L^4_t L^4_x} \|I u_{M_3}\|_{L^4_t L^4_x} \|Iu_{M_4}\|_{L^\infty_t L^\infty_x}.
\]
\begin{lem}\label{lem term 2}
We have 
\begin{align}
\|P_{M_1} I (|u|^2 u)\|_{L^2_t L^2_x} &\lesssim \frac{1}{\scal{M_1}^{k/2}} \|Iu\|^3_{X^{k/2,1/2+}}, \label{estimate 1 term 2} \\
\|I u_{M_j}\|_{L^4_t L^4_x} & \lesssim \frac{1}{\scal{M_j}^{k/2}} \|I u_{M_j}\|_{X^{k/2,1/2+}}, \quad j=2, 3, \label{estimate 2 term 2}\\
\|I u_{M_4}\|_{L^\infty_t L^\infty_x} &\lesssim \|I u_{M_4}\|_{X^{k/2,1/2+}}. \label{estimate 3 term 2}
\end{align}
\end{lem}
\begin{proof}
The estimate $(\ref{estimate 1 term 2})$ is in turn equivalent to
\[
\|\scal{\Lambda}^{k/2} P_{M_1} I(|u|^2 u)\|_{L^2_t L^2_x} \lesssim \|Iu\|^3_{X^{k/2,1/2+}}.
\]
Since $\scal{\Lambda}^{k/2} I$ obeys a Leibniz rule, it suffices to prove
\begin{align}
\|P_{M_1}((\scal{\Lambda}^{k/2} I u_1) u_2 u_3)\|_{L^2_t L^2_x} \lesssim \prod_{j=1}^3 \|I u_j\|_{X^{k/2,1/2+}}. \label{estimate 1 term 2 proof}
\end{align}
The Littlewood-Paley theorem and H\"older's inequality imply
\[
\text{LHS}(\ref{estimate 1 term 2 proof}) \lesssim \|\scal{\Lambda}^{k/2} I u_1\|_{L^4_t L^4_x} \|u_2\|_{L^8_t L^8_x} \|u_3\|_{L^8_t L^8_x}.
\]
We have from Strichartz estimate $(\ref{strichartz estimate})$ that
\[
\|\scal{\Lambda}^{k/2} I u_1\|_{L^4_t L^4_x} \lesssim \|\scal{\Lambda}^{k/2} I u_1\|_{X^{0,1/2+}} = \|Iu_1\|_{X^{k/2,1/2+}}.
\]
Combining Sobolev embedding and Strichartz estimate $(\ref{strichartz estimate})$ yield
\[
\|u_2\|_{L^8_t L^8_x}\lesssim \|\scal{\Lambda}^{k/4} u_2\|_{L^8_t L^{8/3}_x} \lesssim \|\scal{\Lambda}^{k/4} u_2\|_{X^{0,1/2+}} \lesssim \|Iu_2\|_{X^{k/2,1/2+}},
\]
where the last estimate follows from $(\ref{sobolev bound I operator})$. Similarly for $\|u_3\|_{L^8_t L^8_x}$. This shows $(\ref{estimate 1 term 2 proof})$. The estimate $(\ref{estimate 2 term 2})$ follows easily from Strichartz estimate. For $(\ref{estimate 3 term 2})$, we use Sobolev embedding and Strichartz estimate to get
\[
\|I u_{M_4}\|_{L^\infty_t L^\infty_x} \lesssim \|\scal{\Lambda}^{k/2} I u_{M_4}\|_{L^\infty_t L^2_x} \lesssim \|\scal{\Lambda}^{k/2} Iu_{M_4}\|_{X^{0,1/2+}} = \|I u_{M_4}\|_{X^{k/2,1/2+}}. 
\]
The proof is complete.
\end{proof}
We use Lemma $\ref{lem term 2}$ to bound 
\[
B\lesssim \frac{m(M_1)}{m(M_2) m(M_3) m(M_4)}\frac{1}{\scal{M_1}^{k/2} \scal{M_2}^{k/2} \scal{M_3}^{k/2}} \|I u\|_{X^{k/2,1/2+}} \prod_{j=2}^4 \|I u_{M_j}\|_{X^{k/2,1/2+}},
\]
with $M_2 \geq M_3 \geq M_4$ and $M_2 \gtrsim N$. Using $(\ref{modified local well-posedness})$, the estimate $(\ref{estimation of A})$ follows once we have
\begin{align}
\frac{m(M_1)}{m(M_2) m(M_3) m(M_4)}\frac{1}{\scal{M_1}^{k/2} \scal{M_2}^{k/2} \scal{M_3}^{k/2}}  \lesssim N^{-\gamma_0(k)+} M_2^{0-}. \label{term 2}
\end{align}
We now break the frequency interactions into two cases: $M_2 \sim M_3$ and $M_2 \sim M_1$ since $\sum_{j=1}^4 \xi_j=0$. \newline
\textbf{Case 1.} $M_2 \sim M_3, M_2 \geq M_3 \geq M_4$ and $M_2 \gtrsim N$.  
We see that
\begin{align*}
\text{LHS}(\ref{term 2}) &\sim \frac{m(M_1)}{(m(M_2))^2 m(M_4)} \frac{1}{\scal{M_1}^{k/2} \scal{M_2}^k} \lesssim \frac{m(M_1)}{N^{2\alpha(k)} m(M_4) \scal{M_1}^{k/2} \scal{M_2}^{k-2\alpha(k)}} \\
&\lesssim \frac{1}{N^{2\alpha(k)}} \frac{1}{m(M_4) \scal{M_2}^{k-2\alpha(k)}} \lesssim \frac{1}{N^{2\alpha(k)}} \frac{1}{M_2^{k-3\alpha(k)}} \lesssim N^{-(k-\alpha(k))+} M_2^{0-}.
\end{align*}
Here we use that $m(M_2)\scal{M_2}^{\alpha(k)} \geq m(N)N^{\alpha(k)}=N^{\alpha(k)}, m(M_1)\lesssim \scal{M_1}^{k/2}$ and that $m(y)\scal{x}^{\alpha(k)} \gtrsim 1$ for all $1 \leq y \leq x$. \newline
\textbf{Case 2.} $M_2 \sim M_1$, $M_2 \geq M_3 \geq M_4$ and $M_2 \gtrsim N$. We have
\begin{align*}
\text{LHS}(\ref{term 2}) &\lesssim \frac{1}{m(M_3) m(M_4)} \frac{1}{\scal{M_2}^{k} \scal{M_3}^{k/2}} \\
&\lesssim \frac{1}{m(M_3)\scal{M_3}^{\alpha(k)}} \frac{1}{m(M_4)\scal{M_2}^{\alpha(k)}} \frac{1}{\scal{M_2}^{k-\alpha(k)} \scal{M_3}^{k/2-\alpha(k)}} \\
&\lesssim N^{-(k-\alpha(k))+} M_2^{0-}.
\end{align*}
Here we use again $m(M_3)\scal{M_3}^{\alpha(k)}, m(M_4)\scal{M_2}^{\alpha(k)}\gtrsim 1$. By choosing $\alpha(k)$ as in Subcase 3a, we prove $(\ref{term 2})$. The proof of Proposition $\ref{prop almost conservation law}$ is now complete.
\defendproof
\begin{rem} \label{rem comment on gamma}
Let us now comment on the choices of $\alpha(k)$ and $\gamma_0(k)$. As mentioned in Remark $\ref{rem almost conservation law}$, if the increment of the modified energy is $N^{-\gamma_0(k)}$, then we can show (see Section $\ref{section global proof}$, after $(\ref{choice of N})$) that the global well-posedness holds for data in $H^\gamma(\R^k)$ with $\gamma>\frac{k^2}{2(k+\gamma_0(k))}=:\gamma(k)$. We learn from $(\ref{estimate case 2})$ that $\gamma_0(k)\leq k-1/2$, hence $\gamma(k) \geq \frac{k^2}{4k-1}$. On the other hand, in Subcase 3a, we need $\alpha(k) +\gamma-k/2 > 0$ and $\alpha(k)=k-\gamma_0(k)$. Since $\gamma>\gamma(k)$, we have $\alpha(k)+\gamma-k/2> \alpha(k)+\gamma(k)- k/2 \geq \alpha(k) + \frac{k^2}{4k-1} -\frac{k}{2}$. We thus choose $\alpha(k):=\frac{k}{2}-\frac{k^2}{4k-1}= \frac{k(2k-1)}{8k-2}$, hence $\gamma_0(k)=k-\alpha(k)= \frac{k(6k-1)}{8k-2}$.
\end{rem}  
\section{The proof of Theorem $\ref{theorem global existence}$} \label{section global proof}
We now are able to show the global existence given in Theorem $\ref{theorem global existence}$. We only consider positive time, the negative one is treated similarly. The conservation of mass and Lemma $\ref{lem property I operator}$ give
\begin{align}
\|u(t)\|^2_{H^\gamma_x} \lesssim \|Iu(t)\|^2_{H^{k/2}_x} \sim \|Iu(t)\|^2_{\dot{H}^{k/2}_x} + \|Iu(t)\|^2_{L^2_x} \lesssim E(Iu(t))+\|u_0\|^2_{L^2_x}. \label{estimate global proof 1}
\end{align}
By density argument, we may assume that $u_0\in C^\infty_0(\R^k)$. Let $u$ be a global solution to $(\Ntext)$ with initial data $u_0$. As $E(Iu_0)$ is not necessarily small, we will use the scaling $(\ref{scaling})$ to make the energy of rescaled initial data small in order to apply the almost conservation law given in Proposition $\ref{prop almost conservation law}$. Let $\lambda>0$ and $u_\lambda$  be as in $(\ref{scaling})$. We have
\begin{align}
E(Iu_\lambda(0)) = \frac{1}{2}\|Iu_\lambda(0)\|^2_{\dot{H}^{k/2}_x} +\frac{1}{4}\|Iu_\lambda(0)\|^4_{L^4_x}. \label{modified energy scaled initial data}
\end{align}
We then estimate
\[
\|Iu_\lambda(0)\|^2_{\dot{H}^{k/2}_x} \lesssim N^{2(k/2-\gamma)} \|u_\lambda(0)\|^2_{\dot{H}^\gamma_x} = N^{2(k/2-\gamma)} \lambda^{-2\gamma}\|u_0\|^2_{\dot{H}^\gamma_x},
\]
and
\[
\|Iu_\lambda(0)\|^4_{L^4_x}\lesssim \|u_\lambda(0)\|^4_{L^4_x}= \lambda^{-k}\|u_0\|^4_{L^4_x} \lesssim \lambda^{-k}\|u_0\|^4_{H^\gamma_x}.
\]
Note that $\gamma>\gamma(k)\geq k/4$ allows us to use Sobolev embedding in the last inequality. Thus, $(\ref{modified energy scaled initial data})$ gives for $\lambda \gg 1$,
\[
E(Iu_\lambda(0))\lesssim (N^{2(k/2-\gamma)} \lambda^{-2\gamma} + \lambda^{-k}) (1+\|u_0\|_{H^\gamma_x})^4 \leq C_0 N^{2(k/2-\gamma)} \lambda^{-2\gamma}  (1+\|u_0\|_{H^\gamma_x})^4.
\] 
We now choose 
\begin{align}
\lambda:=N^{\frac{k/2-\gamma}{\gamma}} \Big(\frac{1}{2C_0}\Big)^{-\frac{1}{2\gamma}} (1+\|u_0\|_{H^\gamma_x})^{\frac{2}{\gamma}} \label{choosing lambda}
\end{align}
so that $E(Iu_\lambda(0))\leq 1/2$. We then apply Proposition $\ref{prop almost conservation law}$ for $u_\lambda(0)$. Note that we may reapply this proposition until $E(Iu_\lambda(t))$ reaches 1, that is at least $C_1 N^{\gamma_0(k)-}$ times. Therefore, 
\begin{align}
E(Iu_\lambda(C_1 N^{\gamma_0(k)-}\delta)) \sim 1. \label{scaling modified energy}
\end{align}
Now given any $T \gg 1$, we choose $N\gg 1$ so that
\[
T \sim \frac{N^{\gamma_0(k)-}}{\lambda^k} C_1 \delta.
\]
Using $(\ref{choosing lambda})$, we see that 
\begin{align}
T \sim N^{\frac{2(\gamma_0(k)+k)\gamma -k^2}{2\gamma}-}. \label{choice of N}
\end{align}
Here $\gamma>\gamma(k)=\frac{k^2}{2(\gamma_0(k)+k)}$, hence the power of $N$ is positive and the choice of $N$ makes sense for arbitrary $T\gg 1$. Next, using $(\ref{scaling})$, a direct computation shows
\[
E(Iu(t)) = \lambda^k E(I u_\lambda(\lambda^k t)).
\]
Thus, we have from $(\ref{choosing lambda})$, $(\ref{scaling modified energy})$ and $(\ref{choice of N})$ that
\begin{align*}
E(Iu(T)) &= \lambda^k E(Iu_\lambda(\lambda^k T)) = \lambda^k E(I u_\lambda(C_1 N^{\gamma_0(k)-} \delta)) \\
&\sim \lambda^k \leq N^{\frac{k(k/2-\gamma)}{\gamma}} \sim T^{\frac{k(k-2\gamma)}{2(\gamma_0(k)+k)\gamma-k^2}+}.
\end{align*}
This shows that there exists $C_2=C_2(\|u_0\|_{H^\gamma_x},\delta)$ such that
\[
E(Iu(T)) \leq C_2 T^{\frac{k(k-2\gamma)}{2(\gamma_0(k)+k)\gamma-k^2}+},
\]
for $T \gg 1$. This together with $(\ref{estimate global proof 1})$ show that
\[
\|u(T)\|_{H^\gamma_x} \lesssim  C_3 T^{\frac{k(k-2\gamma)}{2(2(\gamma_0(k)+k)\gamma-k^2)}+} + C_4,
\]
where $C_3, C_4$ depend only on $\|u_0\|_{H^\gamma_x}$. The proof of Theorem $\ref{theorem global existence}$ is complete.
\appendix
\section{Linear estimate in $X^{\gamma,b}$ spaces}
In this section, we will give the proof of linear estimates $(\ref{homogeneous estimate})$ and $(\ref{inhomogeneous estimate})$ which is essentially given in \cite{Ginibre}. The estimate $(\ref{homogeneous estimate})$ follows from the fact that 
\begin{align}
\|u\|_{X^{\gamma,b}} = \|e^{-it\Lambda^k} u\|_{H^b_t H^\gamma_x}. \label{equivalent norm}
\end{align} 
Indeed, we have
\[
\| \psi(t) e^{it\Lambda^k} u_0\|_{X^{\gamma,b}} = \|e^{-it\Lambda^k} \psi(t) e^{it\Lambda^k} u_0\|_{H^b_t H^\gamma_x} = \|\psi\|_{H^b_t} \|u_0\|_{H^\gamma_x} \lesssim \|u_0\|_{H^\gamma_x}.
\]
For $(\ref{inhomogeneous estimate})$, we firstly remark that it is a consequence of the following estimate
\begin{align}
\Big\| \psi_\delta(t) \int_0^t g(s) ds \Big\|_{H^b_t} \lesssim \delta^{1-b-b'} \|g\|_{H^{-b'}_t}. \label{inhomogeneous estimate appendix}
\end{align}
In fact, using $(\ref{equivalent norm})$, it suffices to prove
\begin{align}
\Big\|\psi_\delta(t) \int_0^t G(s) ds \Big\|_{H^b_t H^\gamma_x} \lesssim \|G\|_{H^{-b'}_t H^\gamma_x}. \label{equivalent norm application}
\end{align}
We now apply $(\ref{inhomogeneous estimate appendix})$ for $g(s)=\widehat{G}(s,\xi)$ with $\xi$ fixed to have
\begin{align}
\Big\|\psi_\delta(t) \int_0^t \widehat{G}(s,\xi) ds \Big\|_{H^b_t} \lesssim \delta^{1-b-b'} \|\widehat{G}(t,\xi)\|_{H^{-b'}_t}, \label{appendix 1}
\end{align}
where $\widehat{\cdot}$ is the spatial Fourier transform. If we denote 
\[
H(t,x):= \psi_\delta(t) \int_0^t G(s,x) ds,
\]
then $(\ref{appendix 1})$ becomes 
\[
\|\widehat{H}(t,\xi) \|_{H^b_t} \lesssim \delta^{1-b-b'} \|\widehat{G}(t,\xi)\|_{H^{-b'}_t}.
\]
Squaring the above estimate, multiplying both sides with $\scal{\xi}^{2\gamma}$ and integrating over $\R^k$, we obtain $(\ref{equivalent norm application})$. It remains to prove $(\ref{inhomogeneous estimate appendix})$. To do so, we write
\begin{align*}
\psi_\delta(t) \int_0^t g(s) ds & = \psi_\delta(t) \int_\R \Big(\int_0^t e^{i\tau s} ds \Big) \hat{g}(\tau)d\tau = \psi_\delta(t) \int_\R \frac{e^{it\tau}-1}{i\tau} \hat{g}(\tau) d\tau \\
&= \psi_\delta(t) \sum_{k\geq 1} \frac{t^k}{k!} \int_{|\delta \tau|\leq 1} (i\tau)^{k-1} \hat{g}(\tau) d\tau - \psi_\delta(t) \int_{|\delta\tau|\geq 1} (i\tau)^{-1} \hat{g}(\tau) d\tau \\
& +\psi_\delta(t) \int_{|\delta\tau|\geq 1} (i\tau)^{-1} e^{it\tau} \hat{g}(\tau)d\tau =: I + II+III.
\end{align*}
Let us consider the first term. The Cauchy-Schwarz inequality gives
\[
\|I\|_{H^b_t} \leq \sum_{k\geq 1} \frac{1}{k!} \|t^k \psi_\delta\|_{H^b_t} \delta^{1-k} \|g\|_{H^{-b'}_t} \Big(\int_{|\delta \tau| \leq 1} \scal{\tau}^{2b'} d\tau \Big)^{1/2}.
\]
Using that $t^k \psi_\delta(t) =\delta^k \varphi_k(\delta^{-1}t)$ where $\varphi_k(t) = t^k \psi(t)$, we have
\[
\|t^k\psi_\delta\|_{H^b_t} = \delta^k\|\varphi_k(\delta^{-1}t)\|_{H^b_t} = \delta^k \Big( \int_\R \scal{\tau}^{2b} \delta^2|\hat{\varphi}_k(\delta\tau)|^2 d\tau \Big)^{1/2} \lesssim \delta^k \delta^{1/2-b}\|\varphi_k\|_{H^b_t}.
\]  
We also have 
\[
\int_{|\delta\tau|\leq 1} \scal{\tau}^{2b'} d\tau = \int_{|\tau|\leq 1} \scal{\delta^{-1}\tau}^{2b'} \delta^{-1} d\tau \lesssim \delta^{-1-2b'},
\]
since $b'<1/2$. This implies
\[
\|I\|_{H^b_t} \lesssim \sum_{k\geq 1} \frac{1}{k!} \delta^k \delta^{1/2-b} \delta^{1-k} \|g\|_{H^{-b'}_t} \delta^{-1/2-b'} \lesssim \delta^{1-b-b'} \|g\|_{H^{-b'}_t}.
\]
Similarly, we have
\[
\|II\|_{H^b_t} \lesssim \|\psi_\delta\|_{H^b_t} \|g\|_{H^{-b'}_t} \Big( \int_{|\delta\tau|\geq 1} |\tau|^{-2} \scal{\tau}^{2b'} d\tau \Big)^{1/2} \lesssim \delta^{1-b-b'}\|g\|_{H^{-b'}_t},
\]
by using that $\|\psi_\delta\|_{H^b_t} \lesssim \delta^{1/2-b} \|\psi\|_{H^b_t} \lesssim \delta^{1/2-b}$ and
\[
\int_{|\delta\tau|\geq 1} |\tau|^{-2} \scal{\tau}^{2b'}d\tau = \int_{|\tau|\geq 1} |\delta^{-1} \tau|^{-2} \scal{\delta^{-1}\tau}^{2b'} \delta^{-1}d\tau \leq \delta^{1-2b'} \int_{|\tau|\geq 1} |\tau|^{-2} \scal{\tau}^{2b'}d\tau \lesssim \delta^{1-2b'}.
\]
Here $b'<1/2$ hence $2(1-b')>1$ implies the last integral is convergent. We finally treat the third term as follows. Set 
\[
J(t):= \int_{|\delta\tau|\geq 1} (i\tau)^{-1} \hat{g}(\tau) e^{it\tau} d\tau.
\]
We see that
\[
\hat{J}(\zeta)= \int_{|\delta\tau|\geq 1} (i\tau)^{-1} \hat{g}(\tau) \delta_0(\zeta-\tau) d\tau,
\]
where $\delta_0$ is the Dirac delta function. This yields that
\begin{align*}
\|J\|_{H^b_t} = \Big(\int \scal{\zeta}^{2b} |\hat{J}(\zeta)|^2 d\zeta\Big)^{1/2} &= \Big(\int_{|\delta\tau|\geq 1} \scal{\tau}^{2b} |\tau|^{-2} |\hat{g}(\tau)|^2 d\tau\Big)^{1/2} \\
&\leq \|g\|_{H^{-b'}_t} \sup_{|\delta\tau|\geq 1} |\tau|^{-1} \scal{\tau}^{b+b'} \lesssim \delta^{1-b-b'}\|g\|_{H^{-b'}_t}.
\end{align*}
Similarly,
\[
\|J\|_{L^2_t} \lesssim \delta^{1-b'} \|g\|_{H^{-b'}_t}.
\]
Thus, the Young's inequality gives
\[
\|III\|_{H^b_t} = \|\scal{\tau}^b (\hat{\psi}_\delta \star \hat{J})\|_{L^2_\tau} \lesssim \| |\tau|^b \hat{\psi}_\delta\|_{L^1_\tau} \|\hat{J}\|_{L^2_\tau}+\|\hat{\psi}_\delta\|_{L^1_\tau} \|\scal{\tau}^b \hat{J}\|_{L^2_\tau} \lesssim \delta^{1-b-b'} \|g\|_{H^{-b'}_t}.
\]
Here we use the fact that $\scal{\tau}^b \lesssim |\tau-\zeta|^b + \scal{\zeta}^b$ to have the first estimate.  
This completes the proof.
\section*{Acknowledgments}
\addcontentsline{toc}{section}{Acknowledments}
The author would like to express his deep thanks to his wife-Uyen Cong for her encouragement and support. He also would like to thank his supervisor Prof. Jean-Marc BOUCLET for the kind guidance and constant encouragement. He also would like to thank the reviewers for their helpful comments and suggestions, which helped improve the manuscript.


{\sc Institut de Math\'ematiques de Toulouse, Universit\'e Toulouse III Paul Sabatier, 31062 Toulouse Cedex 9, France.} \\
\indent Email: \href{mailto:dinhvan.duong@math.univ-toulouse.fr}{dinhvan.duong@math.univ-toulouse.fr}

\begin{thebibliography}{99}
\addcontentsline{toc}{section}{References}

\bibitem{BCDfourier} {\sc H. Bahouri, J-Y. Chemin, R. Danchin}, {\it Fourier analysis and non-linear partial differential equations}, A Series of Comprehensive Studies in Mathemati\text{s} 343, Springer (2011).


\bibitem{Bourgain-98} {\sc J. Bourgain}, {\it Refinements of Strichartz's inequality and applications to 2D-NLS with critical nonlinearity}, Int. Mat. Res. Not. 5, 253-283 (1998).


\bibitem{CazenaveWeissler} {\sc T. Cazenave, F. B. Weissler}, {\it The Cauchy problem for the critical nonlinear Schr\"odinger equation in $H^s$}, Nonlinear Anal. 14, 807-836 (1990).

\bibitem{Cazenave} {\sc T. Cazenave}, {\it  Semilinear Schr\"odinger equations}, Courant Lecture Notes in Mathemati\text{s} 10, Courant Institute of Mathematical Sciences, AMS (2003).











\bibitem{CoKeStaTaTao-almost-conservation} {\sc J. Colliander, M. Keel, G. Staffilani, H. Takaoka, T. Tao}, {\it Almost conservation laws and global rough solutions to a nonlinear Schr\"odinger equation}, Math. Res. Lett. 9, 659-682 (2002).

\bibitem{CoKeStaTaTao-multilinear} {\sc J. Colliander, M. Keel, G. Staffilani, H. Takaoka, T. Tao}, {\it Multi-linear for periodic KdV equations and applications}, J. Funct. Anal. 211, 173-218 (2004).

\bibitem{CoKeStaTaTao-resonant} {\sc J. Colliander, M. Keel, G. Staffilani, H. Takaoka, T. Tao}, {\it Resonant decompositions and the I-method for cubic nonlinear Schr\"odinger equation on $\R^2$}, Discrete Contin. Dyn. Syst. A 21, 665-686 (2007).

\bibitem{CollianderGrillakisTzirakis} {\sc J. Colliander, M. Grillakis, N. Tzirakis}, {\it Improved interaction Morawetz inequalities for the cubic nonlinear Schr\"odinger equation on $\R^2$}, Int. Math. Res. Not. (2007).


\bibitem{CollianderRoy} {\sc J. Colliander, T. Roy}, {\it Bootstrapped Morawetz estimates and resonant decomposition for low regularity global solutions of cubic NLS in $\mathbb{R}^2$}, Commun. Pure Appl. Anal. 10, No. 2, 397-414 (2011).

\bibitem{Dinh} {\sc V. D. Dinh}, {\it Well-posedness of nonlinear fractional Schr\"odinger and wave equations in Sobolev spaces}, \href{https://arxiv.org/abs/1609.06181}{https://arxiv.org/abs/1609.06181} (2016).

\bibitem{Dinh-fourth-order} {\sc V. D. Dinh}, {\it On well-posedness, regularity and ill-posedness for the nonlinear fourth-order Schr\"odinger equation}, \href{https://arxiv.org/abs/1703.00891}{https://arxiv.org/abs/1703.00891} (2017).

\bibitem{Dodson} {\sc B. Dodson}, {\it Improved almost Morawetz estimates for the cubic nonlinear Schr\"odinger equation}, Comm. Pure Appl. Anal. 10, No. 1, 127-140 (2011).

\bibitem{Dodson-scattering} {\sc B. Dodson}, {\it Global well-posedness and scattering for the defocusing, $L^2$-critical, nonlinear Schr\"odinger equation when $d=2$}, Duke Math. J. 165, No. 18, 3435-3516 (2016). 


\bibitem{FangGrillakis} {\sc Y. Fang, M. Grillakis}, {\it On the global existence of rough solutions of the cubic defocusing Schr\"odinger equation in $\R^{2+1}$}, J. Hyperbolic. Differ. Equ. 4, No. 2, 233-257 (2007).






\bibitem{Ginibre} {\sc J. Ginibre}, {\it Le probl\`eme de Cauchy pour des EDP semi-lin\'eaires p\`eriodiques en variables d'espace (d'apr\'es Bourgain)}, S\'eminaire Bourbaki 1995, Ast\'erisque 237, 163-187 (1996). 


\bibitem{Guo} {\sc C. Guo}, {\it Global existence of solutions for a fourth-order nonlinear Schr\"odinger equation in $n+1$ dimensions}, Nonlinear Anal. 73, 555-563 (2010).






\bibitem{HaoHsiaoWang06} {\sc C. Hao, L. Hsiao, B. Wang}, {\it Well-posedness for the fourth-order Schr\"odinger equations}, J. Math. Anal. Appl. 320, 246–265 (2006).

\bibitem{HaoHsiaoWang07} {\sc C. Hao, L. Hsiao, B. Wang}, {\it Well-posedness of the Cauchy problem for the fourth-order Schr\"odinger equations in high dimensions}, J. Math. Anal. Appl. 328, 58–83 (2007).






\bibitem{Kato95} {\sc T. Kato}, {\it On nonlinear Schr\"odinger equations. II. $H^s$-solutions and unconditional well-posedness}, J. Anal. Math. 67, 281-306 (1995).






  


\bibitem{MiaoWuZhang} {\sc C. Miao, H. Wu, J. Zhang}, {\it Scattering theory below energy for the cubic fourth-order Schr\"odinger equation}, Math. Nachr. 288, No. 7, 798-823 (2015).


\bibitem{Pausader} {\sc B. Pausader}, {\it Global well-posedness for energy critical fourth-order Schr\"odinger equations in the radial case}, Dynamics of PDE 4, No. 3, 197-225 (2007).

\bibitem{Pausadercubic} {\sc B. Pausader}, {\it The cubic fourth-order Schr\"odinger equation}, J. Funct. Anal. 256, 2473-2517 (2009).

\bibitem{PausaderShao} {\sc B. Pausader, S. Shao}, {\it The mass-critical fourth-order Schr\"odinger equation in high dimensions}, J. Hyper. Differential Equations 7, No. 4, 651-705 (2010).






\bibitem{Tao} {\sc T. Tao}, {\it Nonlinear dispersive equations: local and global analysis}, CBMS Regional Conference Series in Mathemati\text{s} 106, AMS (2006).




\end{thebibliography}
\end{document}